\documentclass[12pt,reqno]{amsart}
\usepackage{amssymb,epsf}
\usepackage{amstext,amsmath,amsfonts,amscd,amsthm,graphicx}
\usepackage{upref}
\usepackage{epsfig}
\usepackage[utf8]{inputenc}
\usepackage{latexsym}
\usepackage{eucal}
\usepackage{wrapfig}
\usepackage{color}
\usepackage{enumerate}
\theoremstyle{definition}
\newtheorem{df}{Definition}[section]
\theoremstyle{plain}
\newtheorem{thm}{Theorem}[section]
\newtheorem{lemma}{Lemma}[section]

\theoremstyle{definition}

\newtheorem{cexm}{Counterexample}[section]

\begin{document}
\author{Barbora Voln\'{a}}
\address{Mathematical Institute, Silesian University in Opava, \newline \indent Na Rybn\'{i}\v{c}ku 1, 746 01 Opava, Czech Republic}
\email{Barbora.Volna@math.slu.cz}

\title{On chaotic sets of solutions for a class of differential inclusions on $\mathbb{R}^2$}

\begin{abstract}
We deal with a set of solutions of the continuous multi-valued dynamical systems on $\mathbb{R}^2$ of the form $\dot x \in F(x)$ where $F(x)$ is a set-valued function and $F=\{f_1,f_2\}$. Such dynamical systems are frequently used in mathematical economics. We rectify the sufficient conditions for a set of solutions of this system to exhibit Devaney chaos, $\omega$-chaos and infinite topological entropy from: B.R. Raines, D.R. Stockman, Fixed points imply chaos for a class of differential inclusions that arise in economic models, Trans. American Math. Society 364 (5) (2012), 2479--2492. We significantly improve their results. At the end, we illustrate these problems on our own macroeconomic model.
\end{abstract}

\maketitle

\section{Introduction}

In this paper, we focus on the chaotic behaviour of a set of solutions for particular class of differential inclusions. In general, differential inclusions and dynamical systems given by these differential inclusions are frequently studied, see e.g. \cite{andres_fiser_juttner}, \cite{andres_furst_pastor}, \cite{desheng-li_kloeden}, \cite{desheng-li_xiaoxian-zhang}, \cite{desheng-li_yejuan-wang_suyun-wang}, \cite{obukhovskii_kamenskii_kornev_yeong-cheng-liou}, \cite{smirnov}, and many interesting models of mathematical economics are described by these dynamical systems. The differential inclusion is given by
$$\dot x \in F(x)$$
where $F$ is a set-valued map which associates a set $F(x) \subset \mathbb{R}^n$ to every point $x \in \mathbb{R}^n$ \cite{smirnov}. We consider the particular class of differential inclusions on $\mathbb{R}^2$ with
$$F=\{f_1,f_2\}$$
where $f_{1,2}:\mathbb{R}^2 \rightarrow \mathbb{R}^2$ are $C^1$ functions. The state space $\mathbb{R}^2$ is frequently used in many economic areas where two dimensional models arise, see e.g. \cite{gyurkovics_meyer_takacs}, \cite{ke-ang-fu_chenglong-yu}, \cite{lima} \cite{szydlowski_krawiec}, \cite{yokoo}. The graph of the mapping $F$ can be represented by the union of graphs of the single-valued functions $f_1$ and $f_2$. We can find such a mapping in the context of "backward dynamics" occurring in many economic models, see e.g. \cite{kennedy}, \cite{kennedy_stockman}, \cite{medio_raines}, \cite{palimar_shankar}. In such a multi-valued dynamical system, we can detect the strange property where many different orbits start at the same point from the state space $\mathbb{R}^2$, and also a chaos in a certain sense. 

In this article, we follow up the issues from the paper \cite{raines_stockman} where authors formulated the sufficient conditions for a set of solutions of this differential inclusion to exhibit Devaney chaos, $\omega$-chaos and infinite topological entropy. But we show that these conditions are not sufficient conditions and we reformulate them to be sufficient conditions. Secondly, authors in \cite{raines_stockman} showed that fixed points imply their sufficient conditions. Naturally, we show that fixed points imply the reformulated sufficient conditions. In this paper, we do not deal with Li-Yorke and distributional chaos but the principle is similar. In conclusion, we illustrate these issues on our own macroeconomic equilibrium model with an economic cycle and we demonstrate the mentioned problems on this application in economics.

\section{Main results}
\label{section_main_results} 

In this section, we briefly present our main results. We deal with a set of solutions $D$ of mentioned multi-valued dynamical system. A solution of this differential inclusion is an absolutely continuous function $x: \mathbb{R} \rightarrow \mathbb{R}^2$ such that 
$$\dot x(t) \in F(x(t)) \textrm{ a.e.}$$
where $t \in \mathbb{R}$, \cite{raines_stockman}, \cite{smirnov}. Instead of the action on this differential inclusion, we consider the natural $\mathbb{R}$-action on the set of solutions $D$. So, let
$$T:D \times \mathbb{R} \rightarrow D$$
be the natural $\mathbb{R}$-action on $D$ and $T(x,t)=y=T_t(x)$, $y(s)=x(t+s)$ for all $s \in \mathbb{R}$ \cite{raines_stockman}. Naturally, if $T$ is chaotic on $D$ in a certain sense then the original multi-valued dynamical system $\dot x \in \{ f_1(x),f_2(x) \}$ is considered chaotic in the same sense too.

In order to continue we need to present the Raines-Stockman's sufficient conditions for a set of solutions of this differential inclusion to exhibit Devaney chaos, $\omega$-chaos and infinite topological entropy in the plane region. Let $a,b \in \mathbb{R}^2$. There is a path from $a$ to $b$ generated by $D$ if there exists a solution $v \in D$ and $t_0,t_1 \in \mathbb{R}$ such that $t_0<t_1$ with $v(t_0)=a$ and $v(t_1)=b$ \cite{raines_stockman}. Moreover, if $\dot v$ has finitely many discontinuities on $[t_0,t_1]$ and $a \neq v(t) \neq b$ for all $t_0<t<t_1$ than this path is the simple path \cite{raines_stockman}. The simple path from $a$ to $b$ generated by $D$ is $\{v(t):t_0 \leq t \leq t_1 \} \subseteq \mathbb{R}^2$ oriented in the sense of increasing time and is denoted by $P_{a b}$, or $P_{ab}(t)$ for $t_0 \leq t \leq t_1$, or simply $P$ if it does not cause any confusion. $P_{ab}$ and $P_{ba}$ can be the same subset of $\mathbb{R}^2$ with the reverse orientation. In the following, we use also the notation $P(t_0)=a$ and $P(t_1)=b$. Let $V \subseteq \mathbb{R}^2$ and $V^*=\{x \in D | x(t) \in V \textrm{ for all t} \in \mathbb{R}\}$. Now, we recall the following properties \cite{raines_stockman} denoted by (RS1) and (RS2).
\begin{enumerate}
\item[(RS1)] For every $a,b \in V$ there is a simple path from $a$ to $b$ in $V$ generated by $D$.
\item[(RS2)] There is a solution $w \in D$ such that $w(t) \in V$ for all $t \in \mathbb{R}$ and $\{w(t):t \in \mathbb{R}\}$ is not dense in $V$. 
\end{enumerate}
According to Raines and Stockman \cite{raines_stockman}, (RS1) and (RS2) are the sufficient conditions for $T|_{V^*}$ to exhibit Devaney chaos, and only (RS1) is the sufficient condition for $(D,T)$ to exhibit $\omega$-chaos and infinite topological entropy. Essentially, this means that $V$ can be an arbitrary subset of $\mathbb{R}^2$ with restrictions given by (RS1), or by (RS1) and (RS2), to guarantee chaotic behaviour in our multi-valued dynamical system. There are two problems in such a formulation of the sufficient conditions. Firstly, the set $V$ can not be an arbitrary subset of $\mathbb{R}^2$. And secondly, (RS1), or rather (RS1) and (RS2) are not sufficient conditions because one condition is missing. Thus, our main results are formulated in the following theorems. The proofs are presented in the section \ref{section_proofs}.

\begin{thm}
\label{thm_not_sufficient_conditions}
(RS1) and (RS2) are not sufficient conditions for $T|_{V^*}$ to exhibit Devaney chaos. (RS1) is not sufficient condition for $(D,T)$ to exhibit $\omega$-chaos and for $T$ to have infinite topological entropy.
\end{thm}

Now, we rectify Raines-Stockman's sufficient conditions. Firstly, we provide the following limitation on the set $V$ as a subset of $\mathbb{R}^2$.
\begin{center}
$\mathbb{R}^2 \setminus V$ is one unbounded subset of $\mathbb{R}^2$.
\end{center} 
Secondly, we define the concatenation of paths and then we formulate the new condition denoted by (BV3).

\begin{df}
\label{df_concatenation}
Let $a_1,a_2,a_3,\ldots \in \mathbb{R}^2$. Let $P_{a_1 a_2}$ be the path from $a_1$ to $a_2$ generated by $D$, let $P_{a_2 a_3}$ be the path from $a_2$ to $a_3$ generated by $D$, etc. Let $t_{a_1},t_{a_2},t_{a_3}\ldots \in \mathbb{R}$ be such that $t_{a_1}<t_{a_2}<t_{a_3}<\ldots$ with $P_{a_1 a_2}(t_{a_1})=a_1$, $P_{a_1 a_2}(t_{a_2})=a_2=P_{a_2 a_3}(t_{a_2})$, $P_{a_2 a_3}(t_{a_3})=a_3$, etc. We say that there is a \textit{concatenation} of paths $P_{a_1 a_2}$, $P_{a_2 a_3}$, etc. generated by $D$ provided there exists a solution $\gamma \in D$ such that the path $Q$ fulfilled  
$$
\begin{array}{ll}
Q(t_{a_1})=a_1 ; & \\
Q(t)=P_{a_1 a_2}(t) & \textrm{for } t_{a_1}\leq t \leq t_{a_2} ;\\
Q(t_{a_2})=a_2 ; & \\
Q(t)=P_{a_2 a_3}(t) & \textrm{for } t_{a_2}\leq t \leq t_{a_3} ; \\
Q(t_{a_3})=a_3 ; & \\
\textrm{etc.} &
\end{array}
$$
is generated by $\{\gamma\}$. The \textit{concatenation} of paths $P_{a_1 a_2}$, $P_{a_2 a_3}$, etc. generated by $D$ is the path $Q:=\{\gamma(t):t_{a_1} \leq t \leq \ldots \}$.
\end{df}

\begin{enumerate}
\item[(BV3)] There are all concatenations of simple paths specified in (RS1) generated by $D$.
\end{enumerate}

Finally, we have the following two theorems.

\begin{thm}
\label{thm_sufficient_conditions}
Let $V$ be the subset of $\mathbb{R}^2$ such that $\mathbb{R}^2 \setminus V$ is one unbounded subset of \nolinebreak $\mathbb{R}^2$. Let $V^*=\{x \in D | x(t) \in V \textrm{ for all t} \in \mathbb{R}\}$. Then (RS1) and (BV3) are the sufficient conditions for $T|_{V^*}$ to exhibit Devaney chaos, for $(D,T)$ to exhibit $\omega$-chaos and for $T$ to have infinite topological entropy.
\end{thm}

These reformulated sufficient conditions also ensure Li-Yorke and distributional chaos in our system but in this paper we do not focus on these types of chaos.

\begin{thm}
\label{thm_fixed_points_imply_RS1andBV3}
Fixed points imply the conditions (RS1) and (BV3) with a set $V \subset \mathbb{R}^2$ such that $\mathbb{R}^2 \setminus V$ is one unbounded subset of $\mathbb{R}^2$.
\end{thm}

\section{Proofs of the main results}
\label{section_proofs}

In order to continue with the proofs we describe the set of solutions of our system as a \nolinebreak topological space and we recall the definitions of Devaney chaos, $\omega$-chaos and topological entropy for our system. We consider the metric on the set $D$ defined by 
$$\nu(x,y)=\sup_{t \in \mathbb{R}}\frac{\nu_t(x,y)}{2^{|t|}}$$
where $\nu_t(x,y)=\min\{d(x(t),y(t)),1\}$ and $d(\cdot,\cdot)$ is the usual metric on $\mathbb{R}^2$ \cite{raines_stockman}. So, $D$ is considered as a topological space with the topology generated by this metric. Note that two relevant spaces connected with our multi-valued dynamical system are considered: the state space $\mathbb{R}^2$ and the topological space of solutions $D$. In the set $D$, there can be the solutions of the first branch, i.e. of the differential equation $\dot x = f_1(x)$, the solutions of the second branch, i.e. of the differential equation $\dot x = f_2(x)$, and the solutions constructed by jumping from the integral curves generated by $f_1$ to the integral curves generated by $f_2$, and vice versa, in some points from $\mathbb{R}^2$, see also \cite{raines_stockman}. It is closely related to the modelled problem. We say that such a solution 'switches' from the integral curve generated by $f_1$ to the integral curve generated by $f_2$, and vice versa, in these points. Any such solution is connected with the time sequence of switching from the branch $f_1$ to the branch $f_2$, and vice versa. The other way around, we say that the solution $x$ such that $\dot x(t)=f_1(x(t))$, or $\dot x(t)=f_2(x(t))$, for $t \in [t_a,t_b]$ with $x(t_a)=a$ and $x(t_b)=b$ where $a,b \in \mathbb{R}^2$ 'follows' the integral curve generated by $f_1$, or $f_2$, from $a$ to $b$. Let $R \subseteq D$ be closed and \linebreak $T$-invariant. We say that $(R,T)$ has Devaney chaos, if $T$ is topological transitive, has a \nolinebreak dense set of periodic points and has sensitive dependence on initial conditions on $R$ as usually considered \cite{devaney}, \cite{raines_stockman}. Let $S \subseteq D$ (having at least two points). We say that $S$ is an $\omega$-scrambled set, if for any $x,y \in S$ with $x \neq y$
\begin{itemize}
\item $\omega(x) \setminus \omega(y)$  is uncountable,
\item $\omega(x) \cap \omega(y)$ is not empty,
\item $\omega(x)$ is not included in the set of periodic points
\end{itemize}
where $\omega(x)$ and $\omega(y)$ is the omega-limit set of $x$ and of $y$ under $T$ \cite{raines_stockman}, \cite{shi-hai-li}. The natural $\mathbb{R}$-action $T: D \times \mathbb{R} \rightarrow D$ is called $\omega$-chaotic provided there exists an uncountable \linebreak $\omega$-scrambled set in $D$ \cite{raines_stockman}, \cite{shi-hai-li}. Let $s \in \mathbb{R}^+$ and for each $x,y \in D$ 
$$\nu_s^T(x,y)=\max_{-s \leq t \leq s} \{ \nu(T_t(x),T_t(y)) \}$$
gives the metric on the orbit segments \cite{raines_stockman}. The notion of the topological entropy \cite{katok_hasselblatt} is extended in the usual way for our considered system \cite{stockman} as we can see below. Let $B_T(x,\epsilon,s)=\{y \in D: \nu_s^T(x,y)<\epsilon \}$ denote the open $\epsilon$-neighbourhood around $x \in D$ with respect to the previous metric.  We say that $E \subseteq D$ is $(s,\epsilon)$-spanning provided 
$$D \subseteq \bigcup_{x \in E} B_T(x,\epsilon,s).$$
Let $S_{\nu}(T,\epsilon,s)$ be the minimal cardinality of $(s,\epsilon)$-spanning set. So, the topological entropy of $T$ is defined by 
$$h_{top}(T)=\lim_{\epsilon \rightarrow 0} \limsup_{s \rightarrow \infty} \frac{1}{s} \log{S_{\nu}(T,\epsilon,s)}.$$

Finally, we present the proofs of the main results with comments.

\subsection{Proof of Theorem \ref{thm_not_sufficient_conditions}}

This proof is based on the counterexamples. In Counterexample \ref{cexm_set_of_solutions}, there is an example of the set of solutions where Devaney chaos is not present although (RS1) and (RS2) are fulfilled, and the example of the set of solutions where \linebreak $\omega$-chaos and infinite topological entropy are not present although (RS1) is fulfilled. In these examples, we consider the set $V$ such that $\mathbb{R}^2 \setminus V$ is one unbounded subset of $\mathbb{R}^2$ but the considered sets of solutions do not fulfil newly formulated (BV3). In Counterexample \nolinebreak \ref{cexm_set_V}, we show a subset of $\mathbb{R}^2$ where $\omega$-chaos is not present although the corresponding set of solutions fulfils (RS1) and also newly formulated (BV3).

\begin{cexm}
\label{cexm_set_of_solutions}
The considered subset  $U$ of $\mathbb{R}^2$ is displayed in Figure \ref{fig_counterexamples_set_of_solutions}. The set $U$ is depicted as the line segment with endpoints $c_1$ and $c_2$. The appropriate differential inclusion is denoted by $G=\{g_1,g_2\}$, and $\varphi$, $\psi$ denote the flows generated by $g_1$, $g_2$, respectively.
\begin{figure}[ht]
  \centering
  \includegraphics[height=1.5cm]{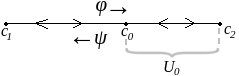}
  \caption{}
  \label{fig_counterexamples_set_of_solutions}
\end{figure}
The arrows represent the trajectories of the flows $\varphi$ and $\psi$. Further, the subset $U_0$ of the set $U$ is depicted as the line segment with endpoints $c_0$ and $c_2$, see Figure \nolinebreak \ref{fig_counterexamples_set_of_solutions}. Now, we focus on the specification of two sets of solutions denoted by $\hat{R}$ and by $\bar{R}$. Firstly, let us consider $X$ as the set of solutions representing the periodic solutions passing through the entire set $U$ (cycle). More precise description of the set $X$ is given in the following. Let $x$ denote an element from $X$ and let $X:=\{x_{a\varphi},x_{a\psi},x_{c_1},x_{c_2}\}$ where $x_{a\varphi}$ and $x_{a\psi}$ represent the sets of solutions with one solution for every $a \in U \setminus \{c_1, c_2 \}$. Each solution $x \in X$ follows the integral curve generated by $g_1$ to the point $c_2$, then $x$ switches to the integral curve generated by $g_2$ and follows this integral curve to the point $c_1$, then in the point $c_1$ the solution $x$ switches to the integral curve generated by $g_1$ and so on for $t \rightarrow \infty$ and also for $t \rightarrow -\infty$, see Figure \ref{fig_counterexamples_set_of_solutions}. So, $x_{a\varphi}$ denotes the solution in $X$ with the initial condition $x_{a\varphi}(0)=a$ for $a \in U \setminus \{c_1, c_2\}$, such that $x_{a\varphi}$ initially follows the integral curve generated by $g_1$ from $a$ to the point $c_2$. Analogously, $x_{a\psi}$ denotes the solution in $X$ with the initial condition $x_{a\psi}(0)=a$ for $a \in U \setminus \{c_1, c_2\}$, such that $x_{a\psi}$ initially follows the integral curve generated by $g_2$ from $a$ to the point $c_1$. $x_{c_1},x_{c_2}$ denotes the solution with the initial condition $x_{c_1}(0)=c_1$, $x_{c_2}(0)=c_2$, respectively. Obviously, $x_{c_1}$ initially follows the integral curve generated by $g_1$ and $x_{c_2}$ the integral curve generated by $g_2$. Thus, let $\ldots<t_{-2}^x<t_{-1}^x<t_0^x<t_1^x<t_2^x<\ldots$ be the time sequence of switching from the branch $g_1$ to the branch $g_2$ in times $t_{2i+1}^x$, $i \in \mathbb{Z}$, and from $g_2$ to $g_1$ in times $t_{2i}^x$, $i \in \mathbb{Z}$, for each $x \in X$. Let $t_0^x<0<t_1^x$ for $x_{a\varphi}$, $t_{-1}^x<0<t_0^x$ for $x_{a\psi}$, $t_0^x=0$ for $x_{c_1}$ and $t_1^x=0$ for $x_{c_2}$. Thus, for each $x \in X$
$$\dot x(t)=g_1(x) \textrm{ for } t \in \bigcup_{i \in \mathbb{Z}}(t_{2i}^x,t_{2i+1}^x)$$
$$\dot x(t)=g_2(x) \textrm{ for } t \in \bigcup_{i \in \mathbb{Z}}(t_{2i-1}^x,t_{2i}^x)$$
with $x(t_{2i}^x)=c_1$ and $x(t_{2i+1}^x)=c_2$ where $i \in \mathbb{Z}$, see Figure \ref{fig_counterexamples_set_of_solutions}. Let $\tau_1:=t_{2i+1}^x-t_{2i}^x$ and $\tau_2:=t_{2i}^x-t_{2i-1}^x$ for each $x \in X$. Note that $\tau_1$ and $\tau_2$ are the same values for all $x$. We see that each solution $x$ is the periodic solution of the length $\tau=\tau_1+\tau_2$, i.e. $x(t+\tau)=x(t)$ for every $t \in \mathbb{R}$. Secondly, let us consider $X_0$ as the set of solutions representing the periodic solutions passing through the entire subset $U_0 \subset U$ (cycle). More precise description of the set $X_0$ is the analogous description as previous with the difference that the end points of the set $U_0$ are $c_0$ and $c_2$, so we write $c_0$ instead of $c_1$ in the previous description. Finally, define $\hat{R}:=X \cup X_0$ and $\bar{R}:=X$. Let $\hat{T}$ denote the natural $\mathbb{R}$-action on $\hat{R}$ and let $\bar{T}$ denote the natural $\mathbb{R}$-action on $\bar{R}$. Clearly, $\hat{R}$, $\bar{R}$ are both closed and $\hat{T}$-, $\bar{T}$-invariant, respectively. The set $U$ evidently has the property that $\mathbb{R}^2 \setminus V$ is one unbounded subset of $\mathbb{R}^2$. And obviously, for every $a,b \in U$ there is a simple path from $a$ to $b$ in $U$ generated by $X$ and there is a solution $w \in X_0$ such that $w(t) \in U$ for all $t \in \mathbb{R}$ and $\{w(t):t \in \mathbb{R}\}$ is not dense in $U$. So, if we consider $U$ with $\hat{R}$ the conditions (RS1) and (RS2) are fulfilled, and if we consider $U$ with $\bar{R}$ the condition (RS1) is fulfilled.
\end{cexm}

\begin{lemma}
Let $U$ be the subset of $\mathbb{R}^2$ specified above and $\hat{R}$ be the set of solutions specified above. Let $U^*=\{y \in \hat{R} | y(t) \in U \textit{ for all t} \in \mathbb{R}\}$. Then $\hat{T}|_{U^*}$ is not chaotic in the sense of Devaney.
\end{lemma}
\begin{proof}
Obviously, $U^*=\hat{R}$. We use the proof by contradiction to show that $\hat{T}|_{U^*}$ has not sensitive dependence on initial conditions, and therefore $\hat{T}|_{U^*}$ is not chaotic in the sense of Devaney. So, let the sensitive dependence on initial conditions be assumed in this system. Let $0<\epsilon_0<1$. Pick $x_{c_1} \in U^*$ with $x_{c_1}(0)=c_1$ specified above. Let $B_{\epsilon_0}(x_{c_1})$ be the open $\epsilon_0$-neighbourhood around $x_{c_1}$ with respect to the $\nu$ metric. Let $\delta>0$ be the sensitivity constant, i.e. for any $x \in U^*$ and $\epsilon>0$ there is a solution $z$ and $s \in \mathbb{R}$ such that $\nu(\hat{T}(x,s),\hat{T}(z,s))>\delta$ and $z \in B_{\epsilon}(x)$. Pick $z_0 \in U^*$ such that $z_0 \in B_{\epsilon_0}(x_{c_1})$ and $\nu(\hat{T}(x_{c_1},s_0),\hat{T}(z_0,s_0))>\delta$ for some $s_0 \in \mathbb{R}$. Let $\delta_0:=\nu(x_{c_1},z_0)$. Since $z_0 \in B_{\epsilon_0}(x_{c_1}) \subseteq U^*$, $x_{c_1}\neq z_0$, $x_{c_1}(0) \neq z_0(0)$ and $\epsilon_0<1$ we see that $d(x_{c_1}(t),z_0(t))\leq d(x_{c_1}(0),z_0(0))<1$ for every $t$ and $0<\delta_0<\epsilon_0$. Thus, we have
$$0<\delta<\nu(\hat{T}(x_{c_1},s_0),\hat{T}(z_0,s_0))\leq \delta_0<\epsilon_0$$
for some $s_0 \in \mathbb{R}$. Now, let $0<\epsilon_1<\delta$. Let $B_{\epsilon_1}(x_{c_1})$ be the open $\epsilon_1$-neighbourhood around $x_{c_1}$ with respect to the $\nu$ metric. Pick $z_1 \in U^*$ such that $z_1 \in B_{\epsilon_1}(x_{c_1})$ and $\nu(\hat{T}(x_{c_1},s_1),\hat{T}(z_1,s_1))>\delta$ for some $s_1 \in \mathbb{R}$. Let $\delta_1:=\nu(x_{c_1},z_1)$. Finally, we have the contradiction
$$\delta<\nu(\hat{T}(x_{c_1},s_1),\hat{T}(z_1,s_1))\leq\delta_1<\epsilon_1<\delta$$
for some $s_1 \in \mathbb{R}$.
\end{proof}

\begin{lemma}
Let $U$ be the subset of $\mathbb{R}^2$ specified above and $\bar{R}$ be the set of solutions specified above. Then $(\bar{R},\bar{T})$ is not $\omega$-chaotic.
\end{lemma}
\begin{proof}
Pick arbitrary $x \in \bar{R}$. Clearly, by construction $\bar{T}(x,\tau)=x$ and for every $y \in \bar{R}$ with $y \neq x$ there exists $r \in (0,\tau)$ such that $y=\bar{T}(x,r)$. And, $\bar{T}(x,r)=\bar{T}(x,r+n\tau)=y$ for $n \in \mathbb{N}$. So, $\lim_{n \rightarrow \infty}\bar{T}(x,r+n\tau)=y$ and thus $y \in \omega(x)$. We see that $\omega(x)=\bar{R}$ for each $x \in \bar{R}$ and $\bar{R}$ contains only periodic points. Thus, there is no $\omega$-scrambled set.
\end{proof}

\begin{lemma}
Let $U$ be the subset of $\mathbb{R}^2$ specified above and $\bar{R}$ be the set of solutions specified above. Then $\bar{T}$ has zero topological entropy.
\end{lemma}
\begin{proof}
Initially, we show that $S_\nu(\bar{T},\epsilon,s)=1$ for $s \geq \frac{\tau}{2}$ and for arbitrary $\epsilon>0$. The orbit of $\bar{T}$ through $x \in \bar{R}$ is $\{\bar{T}_t(x)|t \in \mathbb{R}\}$. Note that $\bar{T}_t(x) \in \bar{R}$ for every $t$. We know that each solution $x \in \bar{R}$ is periodic solution of the length $\tau$. Pick arbitrary $v,w$ from $\bar{R}$. So, we have
$$\{\bar{T}_t(v)|t \in \mathbb{R}\}=\big\{\bar{T}_t(v)|t \in [0,\tau)\big\}=\big\{\bar{T}_t(w)|t \in [0,\tau)\big\}=\{\bar{T}_t(w)|t \in \mathbb{R}\}$$
with $\bar{T}_{0}(v)=\bar{T}_{\tau}(v)$ and $\bar{T}_{0}(w)=\bar{T}_{\tau}(w)$. Let $\{v_i\}_{i \in [0,\tau)}$ and $\{w_j\}_{j \in [0,\tau)}$ denote elements of the orbit of $\bar{T}$ through $v$ and through $w$, respectively. By construction, for every $x,y \in \bar{R}$ there exists $r \in [0,\tau)$ such that $\bar{T}(x,r)=y$. Let $z \in \bar{R}$ and let $\{z_k\}_{k \in [-\frac{\tau}{2},\frac{\tau}{2})}$ be elements from the orbit of $\bar{T}$ through $z$. Thus, for all $v_i$ we find $r_i \in [0,\tau)$ such that $\bar{T}(v_i,r_i)=z_l$ for some $l \in [-\frac{\tau}{2},\frac{\tau}{2})$ and $\{z_l\}_{l \in [-\frac{\tau}{2},\frac{\tau}{2})}$ contains all elements of the orbit of $\bar{T}$ through $z$. Analogously, for all $w_j$ we find $r_j \in [0,\tau)$ such that $\bar{T}(w_j,r_j)=z_m$ for some $m \in [-\frac{\tau}{2},\frac{\tau}{2})$ and $\{z_m\}_{m \in [-\frac{\tau}{2},\frac{\tau}{2})}$ contains all elements of the orbit of $\bar{T}$ through $z$. Obviously, $\{z_l\}_{l \in [-\frac{\tau}{2},\frac{\tau}{2})}=\{z_m\}_{m \in [-\frac{\tau}{2},\frac{\tau}{2})}$. So, we see that
$$\nu_s^{\bar{T}}(v,w)=\max_{-s \leq t \leq s} \{ \nu(\bar{T}_t(v),\bar{T}_t(w)) \}=0$$
for $s \geq \frac{\tau}{2}$ and for arbitrary $v,w \in \bar{R}$.
From this it follows that $(s,\epsilon)$-spanning $\bar{E}=\{x\}$ where $x \in \bar{R}$ arbitrary, and $B_{\bar{T}}(x,\epsilon,s)=\bar{R}$ for $s \geq \frac{\tau}{2}$ and for arbitrary $\epsilon>0$. Thus, the minimal cardinality of $(s,\epsilon)$-spanning $\bar{E}$ is $S_\nu(\bar{T},\epsilon,s)=1$ for $s \geq \frac{\tau}{2}$ and for arbitrary $\epsilon>0$. Finally,
$$h_{top}(\bar{T})=\lim_{\epsilon \rightarrow 0} \limsup_{s \rightarrow \infty} \frac{1}{s} \log{S_{\nu}(\bar{T},\epsilon,s)}=\lim_{\epsilon \rightarrow 0} \limsup_{s \rightarrow \infty} \frac{1}{s} \log{1}=0.$$
\end{proof}

\begin{cexm}
\label{cexm_set_V}
The considered subset $W$ of $\mathbb{R}^2$ is displayed.in Figure \ref{fig_counterexample_set_V}. The set $W$ is depicted as two curved lines connected in the point $c_1$ and in the point $c_2$. The appropriate differential inclusion is denoted by  $H=\{h_1,h_2\}$, and $\varphi$, $\psi$ denote the flows generated by $h_1$, $h_2$, respectively.
\begin{figure}[ht]
  \centering
  \includegraphics[height=2.5cm]{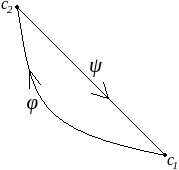}
  \caption{}
  \label{fig_counterexample_set_V}
\end{figure}
The arrows represent trajectories of the flows $\varphi$ and \nolinebreak $\psi$. The set of solutions denoted by $\tilde{S}$ is specified in the following. Let $x_a$ denote an element from $\tilde{S}$ such that $x_a(0)=a$. Each solution $x_a$ passes through the entire set $W$ and alternately follows the integral curve generated by $h_1$ and the integral curve generated by $h_2$. In the point $c_1$ the solution $x_a$ switches from the integral curve generated by $h_2$ to the integral curve generated by $h_1$, and in the point $c_2$ from the integral curve generated by $h_1$ to the integral curve generated by $h_2$. Obviously, each $x_a$ is the periodic solution. Let $\tilde{T}$ denote the natural $\mathbb{R}$-action on $\tilde{S}$. $\tilde{S}$ is closed and $\tilde{T}$-invariant. Evidently, for every $a,b \in W$ there is a simple path from $a$ to $b$ in $W$ generated by $\tilde{S}$ and there are all concatenations of simple path generated by $\tilde{S}$, so the conditions (RS1) and (BV3) are fulfilled. But we can see that $\tilde{S}$ contains only periodic solutions and so $(\tilde{S},\tilde{T})$ is not $\omega$-chaotic.
\end{cexm}

\subsection{Proof of Theorem \ref{thm_sufficient_conditions}}

In all Raines-Stockman's proofs, authors automatically did concatenations of simple paths. This means that the condition (BV3) was used besides conditions (RS1) and (RS2) in their proofs. From this it follows that (BV3) is one of the sufficient conditions for $T|_{V^*}$ to exhibit Devaney chaos, for $(D,T)$ to exhibit $\omega$-chaos and for $T$ to have infinite topological entropy. Additionally, Raines and Stockman considered only the set $V$ with the property that $\mathbb{R}^2 \setminus V$ is one unbounded subset of $\mathbb{R}^2$. So, this property of the set $V$ as the subset of $\mathbb{R}^2$ belongs to the formulation of these sufficient conditions. Now, we show that (RS2) follows from (BV3) and (RS1) for such a set $V$ (with this property).

\begin{lemma}
(RS1) $\&$ (BV3) $\Rightarrow$ (RS2)
\end{lemma}
\begin{proof}
We can see that in each set $V$ complying with our assumptions there exist $\bar{a},\bar{b},\bar{c} \in V$ and $\epsilon>0$ such that $D_{\epsilon}(\bar{c}) \subset V$ and $P_{\bar{a}\bar{b}}$ and $P_{\bar{b}\bar{a}}$ do not intersect $D_{\epsilon}(\bar{c})$ where $D_{\epsilon}(\bar{c})$ is the open $\epsilon$-neighbourhood around $\bar{c}$ with respect to the usual metric on $\mathbb{R}^2$, $P_{\bar{a}\bar{b}}$ is the simple path from $\bar{a}$ to $\bar{b}$ in $V$ generated by $D$ and $P_{\bar{b}\bar{a}}$ is the simple path from $\bar{b}$ to $\bar{a}$ in $V$ generated by $D$. The possible cases are sketched in Figure \ref{fig_proof_BV3_1} and \ref{fig_proof_BV3_2}.
\begin{figure}[ht]
  \centering
  \includegraphics[height=1.7cm]{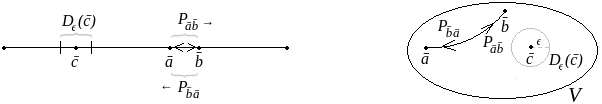}
  \caption{}
  \label{fig_proof_BV3_1}
\end{figure}
In Figure \ref{fig_proof_BV3_1}, we can see the subset of $V$ with empty interior sketched by the line segment (in the left scheme) and the set $V$ with non-empty interior (in the right scheme) where points $\bar{a}$, $\bar{b}$ and $\bar{c}$ with $D_{\epsilon}(\bar{c})$ are located. In this figure, the paths $P_{\bar{a}\bar{b}}$ and $P_{\bar{b}\bar{a}}$ are sketched by the same curve segment with arrows representing the orientation of $P_{\bar{a}\bar{b}}$ and $P_{\bar{b}\bar{a}}$ in the sense of increasing time.
\begin{figure}[ht]
  \centering
  \includegraphics[height=1.7cm]{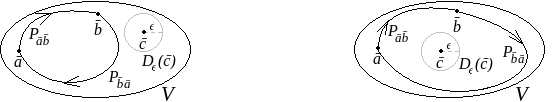}
  \caption{}
  \label{fig_proof_BV3_2}
\end{figure}
In Figure \ref{fig_proof_BV3_2}, we can see the set $V$ with non-empty interior, points $\bar{a}$, $\bar{b}$, paths $P_{\bar{a}\bar{b}}$ and $P_{\bar{b}\bar{a}}$ and the point $\bar{c}$ with $\epsilon$-neighbourhood $D_{\epsilon}(\bar{c})$ located outside (displayed in the left scheme) or inside (displayed in the right scheme) the region bordered by $P_{\bar{a}\bar{b}}$ and $P_{\bar{b}\bar{a}}$. In each case, we construct the solution $w \in D$ so that $w$ alternately follows the simple paths $P_{\bar{a}\bar{b}}$ and $P_{\bar{b}\bar{a}}$. Such a solution exists in $D$ because of our assumptions (simple paths and concatenations). Let $t_{\bar{a}\bar{b}}>0$ and $t_{\bar{b}\bar{a}}>0$ be defined such that $P_{\bar{a}\bar{b}}(t_{\bar{a}\bar{b}})=\bar{b}$ and $P_{\bar{b}\bar{a}}(t_{\bar{b}\bar{a}})=\bar{a}$. So, we define $w$ by
$$
\begin{array}{ll}
w(t)=a & \textrm{for } t=k t_{\bar{a}\bar{b}}+k t_{\bar{b}\bar{a}} ; \\
w(t)=P_{\bar{a}\bar{b}}(t) & \textrm{for } k t_{\bar{a}\bar{b}}+k t_{\bar{b}\bar{a}} \leq t \leq (k+1)t_{\bar{a}\bar{b}}+k t_{\bar{b}\bar{a}} ;\\
w(t)=b & \textrm{for } t=(k+1)t_{\bar{a}\bar{b}}+k t_{\bar{b}\bar{a}} ; \\
w(t)=P_{\bar{b}\bar{a}}(t) & \textrm{for } (k+1)t_{\bar{a}\bar{b}}+k t_{\bar{b}\bar{a}} \leq t \leq (k+1)t_{\bar{a}\bar{b}}+(k+1)t_{\bar{b}\bar{a}} ;
\end{array}
$$
for $k \in \mathbb{N}_0$. We know that $\{w(t):t \in \mathbb{R}\} \cap D_{\epsilon}(\bar{c})=\emptyset$ and $D_{\epsilon}(\bar{c})$ is open. Since $\{w(t):t \in \mathbb{R}\} \subseteq V \setminus D_{\epsilon}(\bar{c})$ the solution $w$ is not dense in $V$.
\end{proof}

\subsection{Proof of Theorem \ref{thm_fixed_points_imply_RS1andBV3}}

Raines and Stockman \cite{raines_stockman}, \cite{stockman_raines} showed that for differential inclusion $F=\{f_1,f_2\}$ with the properties that
\begin{itemize}
\item $f_1$ has a hyperbolic singular point $a^*$ in a region where $f_2$ has no bounded solutions,
\item and $a^*$ is a sink, or a source, or also a saddle point with requirement that $f_2(a^*)$ is not a scalar multiple of an eigenvector of $Df_1(a^*)$,
\item the solutions from $D$ have no restrictions of switching from integral curves generated by one branch of $F$ to the integral curves generated by the other branch of \nolinebreak $F$,
\end{itemize}
we can construct a set $V$ in the following way. Let $K \subset \mathbb{R}^2$ denote the non-empty compact set, where $f_2$ has no bounded solutions. Let $a \in K$. Let $P$ denote the simple path from $a$ to $a$ generated by $D$ such that $P \subset K$. Thus, $P$ is a finite union of arcs and is compact. $\mathbb{R}^2 \setminus P$ has only one unbounded component denoted by $C_0$. The set $\mathbb{R}^2 \setminus P$ can be written as $\Big(\bigcup_{\alpha \in A}C_{\alpha}\Big) \cup C_0$ where $C_{\alpha}$ are bounded components. So, the set $V$ is given by
\begin{equation}
\label{eq_construction_of_V}
\Big(\bigcup_{\alpha \in A}C_{\alpha}\Big) \cup P,
\end{equation}
In \cite{raines_stockman} authors showed that such constructed set $V$ (\ref{eq_construction_of_V}) is closed and satisfies (RS1) and (RS2). In \cite{volna}, we showed that in the last case where $f_2(a^*)$ is a scalar multiple of an eigenvector of $Df_1(a^*)$ a set $V$ given by (\ref{eq_construction_of_V}) can be also constructed. Note, the region $V$ is near the singular point $a^*$ because of the local character of hyperbolic singular points, generally. Now, we show that such constructed set $V$ (\ref{eq_construction_of_V}) satisfies the condition (RS1) and the condition (BV3), see the following paragraphs \ref{paragraph_RS1}, \ref{paragraph_BV3}. By construction, $\mathbb{R}^2 \setminus V$ is one unbounded subset of $\mathbb{R}^2$.

\begin{enumerate}[I]
\item \label{paragraph_RS1}
We show that the set $V$ given by (\ref{eq_construction_of_V}) meets the condition (RS1), similarly as in \nolinebreak \cite{raines_stockman}. But Raines and Stockman constructed a path in $V$ between two arbitrary points from $V$ which may not be the simple path. Let $c,d \in V$ arbitrary. We can find a simple path in $V$ from $c$ to some $c_P \in P$ and a simple path in $V$ from some $d_P \in P$ to $d$ (because $K$ is a region where the branch $f_2$ has no bounded solutions). Thus, the main idea from \cite{raines_stockman} relies on the reasoning that we can connect four simple paths - $P_{c c_P}$, $P_{c_P a}$, $P_{a d_P}$, $P_{d_P d}$ - and the resulting path is the simple path in $V$. The path $P_{c c_P}$ obviously belongs to the branch $f_2$, $P_{c_P a}$ and $P_{a d_P}$ are subsets of $P$, and $P_{d_P d}$ belongs to the branch $f_2$. Such a connection of simple paths is possible because in the points $a, c_P, d_P$ the corresponding solutions switches from the original integral curve to the other integral curve, or follows the original integral curve (no restrictions of switching). But this path constructed in \cite{raines_stockman} may not be the simple path if $c \in P \, \& \, d \not\in P$, or $c \not\in P \, \& \, d \in P$, or $c,d \in P$. In Figure \ref{fig_simple_paths_and_concatenations}, there are the set $V \subset K$ with non-empty interior (in the left scheme) and with empty interior (in the right scheme), the simple path $P$ from $a$ to $a$ with arrows displaying the orientation of $P$ in the sense of increasing time and the trajectories of the flow $\psi$ generated by $f_2$ represented by arrows passing through $K$. If we consider $V$ with non-empty interior then for example the path $P_{c_1 d_2}$ connects the simple paths $P_{c_1 c_2}$ (where $c_2=c_P$), $P_{c_2 a}$ and $P_{a d_2}$, and the resulting path is not the simple path, or the path $P_{c_2 d_2}$ connects the simple paths $P_{c_2 a}$ and $P_{a d_2}$, and the resulting path is not the simple path (see the left scheme in Figure \ref{fig_simple_paths_and_concatenations}). Analogously for the $V$ with empty interior, for example the path $P_{c_1 d_1}$( where $c_1=c_P$ and $d_1=d_P$) connecting simple paths $P_{c_1 a}$ and $P_{a d_1}$ is not simple (see the right scheme in Figure \ref{fig_simple_paths_and_concatenations}). Next to this our main idea relies on the reasoning that we can use the subset of $P$ from $c_P$ to $d_P$ which does not have to go through the point $a$ but such that the resulting path is simple. And so, for our previous examples for $V$ with non-empty interior, the simple path $P_{c_1 d_2}$ connects the simple paths $P_{c_1 c_2}$ and $P_{c_2 d_2}$, the simple path $P_{c_2 d_2}$ is the subset of $P$ not containing the point $a$ (see the left scheme in Figure \ref{fig_simple_paths_and_concatenations}). For our previous example for $V$ with empty interior, the simple path $P_{c_1 d_1}$ is the subset of $P$ not containing the point $a$ (see the right scheme in Figure \ref{fig_simple_paths_and_concatenations}). And so, using this way we can construct the simple path in $V$ between two arbitrary points from $V$ given by (\ref{eq_construction_of_V}).
\begin{figure}[ht]
  \centering
  \includegraphics[height=3cm]{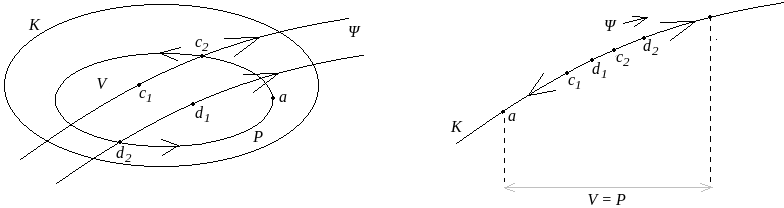}
  \caption{}
  \label{fig_simple_paths_and_concatenations}
\end{figure}
\item \label{paragraph_BV3}
We show that the set $V$ given by (\ref{eq_construction_of_V}) meets the condition (BV3). Let $a,b,c \in V$ arbitrary. We say that the point $b$ is the connection point of paths $P_{ab}$ and $P_{bc}$. There exist two types of concatenations of two simple paths constructed in the way described in the paragraph \ref{paragraph_RS1}: the solution corresponding to the concatenation follows the same integral curve through the connection point or switches from the integral curve generated by the original branch to the integral curve generated by the other branch in the connection point (no restrictions of switching). We can see this using the schemes in Figure \ref{fig_simple_paths_and_concatenations} for the simple paths $P_{ij}$ and $P_{jk}$ where $i,j,k=c_1,d_1,c_2,d_2$. The construction of a concatenation of $n$ simple paths for $n \geq 3$ is similar and is done by the same method. And so, there are all concatenation of simple paths specified in (RS1) generated by $D$ constructed in the mentioned way.
\end{enumerate}

\section{Conclusion and application in economics}

In this paper, we deal with the problem of chaos existence for a class of differential inclusion in $\mathbb{R}^2$ which is implied by fixed points. We present several counterexamples and supplements to \cite{raines_stockman} and we correct the mentioned problems. Finally, we can say that fixed points imply chaos for this class of differential inclusion but Raines-Stockman's proofs are rectified.

Here, we provide our economic application and illustrations of the mentioned problems. A macroeconomic model with an economic cycle can be an example of our differential inclusion $\dot x \in \{f_1(x),f_2(x)\}$ with two branches $f_1$ and $f_2$. The economic cycle (or the business cycle) consists of the expansion phases and of the recession phases which alternate, see \cite{hamilton}, \cite{motohiro_yogo}, \cite{orlando}. In the peak points the recession replaces the expansion and in the trough points the expansion replaces the recession \cite{hamilton}, \cite{motohiro_yogo}, \cite{orlando}. The first branch $\dot x=f_1(x)$ represents the description of macroeconomic situation in a recession and the second branch $\dot x=f_2(x)$ represents the description of macroeconomic situation in an expansion. The solution fulfilling $\dot x(t) \in \{f_1(x(t)),f_2(x(t))\}$ a.e. and switching from the integral curve generated by $f_1$ to the integral curve generated by $f_2$ and vice versa  represents the alternation of these phases in an economy. The trough points are represented by points where this solution switches from the integral curve generated by $f_1$ to the integral curve generated by $f_2$, and vice versa for the peak points. In \cite{volna}, we created such a \nolinebreak model - the branch belonging to the recession phases is the well-known macroeconomic equilibrium model called IS-LM model \cite{gandolfo}, \cite{hicks}, and the branch belonging to the expansion phases is the newly created model called QY-ML model \cite{volna}. This model is named IS-LM/QY-ML model \cite{volna} and can be briefly described by
$$
\left(
\begin{array}{c}
\dot Y \\
\dot R 
\end{array}
\right)
\in \left\{
\left(
\begin{array}{c}
\alpha_1 [I(Y,R)-S(Y,R)] \\
\beta_1 [L(Y,R)-M(Y,R)
\end{array}
\right),
\left(
\begin{array}{l}
\alpha_2 [Q(Y,R)-Y] \\
\beta_2 [M(Y,R)-L(Y,R)]
\end{array}
\right)
\right\}
$$
where $\alpha_1,\alpha_2,\beta_1,\beta_2>0$ and \\
\begin{tabular}{lp{10cm}}
     $Y$  & is an aggregate income (GDP, GNP), \\
     $R$  & is an interest rate, \\
     $I$  & is an investment, \\
     $S$  & is a saving, \\
     $Q$  & is a production, \\
     $L$  & is a demand for money, \\
     $M$  & is a supply of money. \\    
\end{tabular} \\

Firstly, let us consider the set $U \subset \mathbb{R}^2$ and the sets of solutions $\bar{R}$ and $\hat{R}$ described in Counterexample \ref{cexm_set_of_solutions} occurring in the system given by this model. The set of solutions $\bar{R}$, or $\hat{R}$, can be interpreted so that there exists one regular economic cycle in an economy, or there exist two regular economic cycles with different period. In this situation, the solutions have  restrictions of switching between branches and this restriction follows from the modelled economic problem. Secondly, let us consider the set $W \subset \mathbb{R}^2$ and the set of solutions $\tilde{S}$ described in Counterexample \ref{cexm_set_V} occurring in the system given by this model. The interpretation is such that the set $W$ can occur in an economy with some statutory limitations on the levels of the aggregate income $Y$ and of the interest rate $R$ (levels belonging to the 'interior' and 'exterior' are forbidden). This situation requires an additional restriction on the range of solutions, i.e. solutions are absolutely continuous functions $x:\mathbb{R} \rightarrow K_1 \subset \mathbb{R}^2$ with $W \subseteq K_1$, and this restriction follows from the modelled economic problem. And finally, let us consider that properties of relevant economic quantities are such that there exist chaos in mentioned sense in the system given by this model (see more in \cite{volna}). This case can be interpreted so that there exist economic cycles with all possible periods and lengths of the recession and expansion phases and there are no statutory limitations on the levels of the aggregate income $Y$ and of the interest rate $R$.

\end{document}